\documentclass{amsart}
\usepackage{amsfonts}
\usepackage{amssymb}
\usepackage{amsmath,amsxtra,amssymb,mathrsfs,amscd}
\numberwithin{equation}{section}
\usepackage[all]{xy}
\usepackage{bbm} %% used to use \mathbbm{1}

\newtheorem{thm}{\textbf{Theorem}}[section]
\newtheorem{prop}[thm]{\textbf{Proposition}}

\newtheorem{lemma}[thm]{\textbf{Lemma}}

\theoremstyle{definition}
\newtheorem{example}[thm]{\textbf{Example}}

\theoremstyle{definition}
\newtheorem{defi}[thm]{\textbf{Definition}}

\theoremstyle{definition}
\newtheorem{remark}[thm]{\textbf{Remark}}

\theoremstyle{definition}

\theoremstyle{definition}
\newtheorem{question}[thm]{\textbf{Question}}

\newcommand{\bp}{\mathcal{B}}
\newcommand{\lpw}{L_p(\Omega)}

\newcommand{\cbp}{\mathcal{CB}_p}

\newcommand{\llangle}{\langle\!\langle}
\newcommand{\rrangle}{\rangle\!\rangle}
\newcommand{\supp}{\operatorname{supp}}

\def\proclaim #1. #2\par{\medbreak
\noindent{\bf#1.\enspace}{\sl#2}\par\medbreak}
\allowdisplaybreaks

\begin{document}

%============================
%  TITLE AND AUTHOR(s)
%============================
\title[Hahn-Banach Type Extension Theorems on $p$-Operator Spaces]
{Hahn-Banach Type Extension Theorems on $p$-Operator Spaces}
\author{Jung-Jin Lee}
\address{Department of Mathematics and Statistics\\
Mount Holyoke College, South Hadley, MA 01075, USA} \email[Jung-Jin
Lee]{jjlee@mtholyoke.edu}
%\author{Zhong-Jin Ruan}
%\address{Department of Mathematics\\
%University of Illinois, Urbana, IL 61801, USA} \email[Zhong-Jin
%Ruan]{ruan@math.uiuc.edu}
%\thanks{${}^*$ The first author was supported
%by the Innovation Foundation of Nankai University
%The Project Sponsored by the Scientific Research Foundation for
%the Returned Overseas Chinese Scholars, State Education Ministry.
%by the Project Sponsored by SRF for ROCS, SEM,
%and the  third author was partially supported by the National Science
%Foundation DMS-0901395. }
\thanks{The author was supported by Hutchcroft Fund, Department of Mathematics and Statistics, Mount Holyoke College}

%============================
%  GENERAL INFORMATION
%============================

\date{\today}

%============================
%  ABSTRACT
%============================

\begin{abstract} Let $V\subseteq W$ be two operator spaces. Arveson-Wittstock-Hahn-Banach theorem asserts that every completely contractive map $\varphi:V\to \mathcal{B}(H)$ has a completely contractive extension $\tilde{\varphi}:W\to \mathcal{B}(H)$, where $\mathcal{B}(H)$ denotes the space of all bounded operators from a Hilbert space $H$ to itself. In this paper, we show that this is not in general true for $p$-operator spaces, that is, we show that there are $p$-operator spaces $V\subseteq W$, an $SQ_p$ space $E$, and a $p$-completely contractive map $\varphi:V\to \mathcal{B}(E)$ such that  $\varphi$ does not extend to a $p$-completely contractive map on $W$. Restricting $E$ to $L_p$ spaces, we also consider a condition on $W$ under which every completely contractive map $\varphi:V\to \mathcal{B}(L_p(\mu))$ has a completely contractive extension $\tilde{\varphi}:W\to \mathcal{B}(L_p(\mu))$.
\end{abstract}

\maketitle

\section{Introduction to $p$-Operator Spaces}
Throughout this paper, we assume $1< p,p' < \infty$ with $1/p+1/p'=1$, unless stated otherwise. For a Banach space $X$, we denote by $\mathbb{M}_{m,n}(X)$ the linear space of all $m\times n$ matrices with entries in $X$. By $\mathbb{M}_n(X)$, we will denote $\mathbb{M}_{n,n}(X)$. When $X=\mathbb{C}$, we will simply use $\mathbb{M}_{m,n}$ (respectively, $\mathbb{M}_n$) for $\mathbb{M}_{m,n}(\mathbb{C})$ (respectively, $\mathbb{M}_{n}(\mathbb{C})$). For Banach spaces $X$ and $Y$, we will denote by $\mathcal{B}(X,Y)$ the space of all bounded linear operators from $X$ to $Y$. We will also use $\mathcal{B}(X)$ for $\mathcal{B}(X,X)$. The $\ell_p$ direct sum of $n$ copies of $X$ will be denoted by $\ell_p^n(X)$.

\begin{defi} Let $SQ_p$ denote the collection of subspaces of quotients of $L_p$ spaces. A Banach space $X$ is called a \textit{concrete
$p$-operator space} if $X$ is a closed subspace of $\mathcal{B}(E)$ for some $E\in SQ_p$.
\end{defi}

Let $E\in SQ_p$. For a concrete $p$-operator space $X\subseteq\mathcal{B}(E)$ and for each $n\in \mathbb{N}$, define a norm $\|\cdot\|_n$ on
$\mathbb{M}_n(X)$ by identifying $\mathbb{M}_n(X)$ as a subspace of $\mathcal{B}(\ell_p^n(E))$, and let $M_n(X)$ denote the corresponding normed space. The norms $\|\cdot\|_n$ then satisfy

\begin{description}
\item[$\mathcal{D}_\infty$] for $u \in M_n(X)$ and $v \in M_m(X)$, we have $\|u\oplus v\|_{M_{n+m}(X)}=\max\{\|u\|_n,\|v\|_m\}$.
\item[$\mathcal{M}_p$] for $u \in M_m(X)$, $\alpha \in \mathbb{M}_{n,m}$, and $\beta \in \mathbb{M}_{m,n}$, we have
$\|\alpha u \beta\|_n\leq \|\alpha\|\|u\|_m\|\beta\|$, where $\|\alpha\|$ is the norm of $\alpha$ as a member of $\mathcal{B}(\ell_p^m,\ell_p^n)$, and similarly for $\beta$.
\end{description}

When $p=2$, these are Ruan's axioms and $2$-operator spaces are simply operator spaces because the $SQ_2$ spaces are exactly the same as Hilbert spaces.

As in operator spaces, we can also define abstract $p$-operator spaces.

\begin{defi}
An \textit{abstract $p$-operator space} is a Banach space $X$ together with a sequence of norms $\|\cdot\|_n$ defined on $\mathbb{M}_n(X)$ satisfying the conditions $\mathcal{D}_\infty$ and $\mathcal{M}_p$ above.
\end{defi}

Thanks to Ruan's representation theorem\cite{Ruan}, we do not distinguish between concrete and abstract operator spaces. Le Merdy showed that this remains true for $p$-operator spaces.

\begin{thm}\cite[Theorem 4.1]{LeMerdy} \label{LeMerdy characterization} An abstract $p$-operator space $X$ can be isometrically embedded in $\mathcal{B}(E)$ for some $E\in SQ_p$ in such a way that the canonical norms on $\mathbb{M}_n(X)$ arising from this embedding agree with the given norms.
\end{thm}

\begin{example} \label{examplefirst}\quad
\begin{enumerate}
\item Suppose $E$ and $F$ are $SQ_p$ spaces and let $L=E\oplus_p F$, the $\ell_p$ direct sum of $E$ and $F$. Then $L$ is also an $SQ_p$ space  \cite[Proposition 5]{Herz} and the mapping
$$x\mapsto \left[\begin{array}{cc} 0 & 0 \\ x & 0 \end{array}\right]$$
is an isometric embedding of $\mathcal{B}(E,F)$ into $\mathcal{B}(L)$. Using this we can view $\mathcal{B}(E,F)$ as a $p$-operator space.
Note that $M_n(\mathcal{B}(E,F))$ is isometrically isomorphic to $\mathcal{B}(\ell_p^n(E),\ell_p^n(F))$.
\item The identification $L_p(\mu)=\mathcal{B}(\mathbb{C},L_p(\mu))\subseteq \mathcal{B}(\mathbb{C}\oplus_pL_p(\mu))$ gives a $p$-operator space structure on $L_p(\mu)$ called the \textit{column p-operator space structure} of $L_p(\mu)$, which we denote by $L_p^c(\mu)$. Similarly, the identification $L_{p'}(\mu)=\mathcal{B}(L_p(\mu),\mathbb{C})$ gives rise to $p$-operator space structure on $L_{p'}(\mu)$ which we denote by $L_{p'}^r(\mu)$ and call the \textit{row p-operator space structure} of $L_{p'}(\mu)$. In general, we can define $E^c$ and $(E')^r$ for any $E\in SQ_p$, where $E'$ is the Banach dual space of $E$.
\end{enumerate}
\end{example}

Note that a linear map $u:X\to Y$ between $p$-operator spaces $X$ and $Y$ induces a map $u_n:M_n(X)\to M_n(Y)$ by applying $u$ entrywise. We say that $u$ is $p$-\textit{completely bounded} if $\|u\|_{pcb}:=\sup_n\|u_n\|<\infty$. Similarly, we define $p$-\textit{completely contractive}, $p$-\textit{completely isometric}, and $p$-\textit{completely quotient} maps. We write $\mathcal{CB}_p(X,Y)$ for the space of all $p$-completely bounded maps from $X$ into $Y$.

To turn the mapping space $\mathcal{CB}_p(X,Y)$ between two $p$-operator spaces $X$ and $Y$ into a $p$-operator space, we define a norm on $\mathbb{M}_n(\mathcal{CB}_p(X,Y))$ by identifying this space with $\mathcal{CB}_p(X,M_n(Y))$. Using Le Merdy's theorem, one can show that $\mathcal{CB}_p(X,Y)$ itself is a $p$-operator space. In particular, the $p$-\textit{operator dual space} of $X$ is defined to be $\mathcal{CB}_p(X,\mathbb{C})$. The next lemma by Daws shows that we may identify the Banach dual space $X'$ of $X$ with the $p$-operator dual space $\mathcal{CB}_p(X,\mathbb{C})$ of $X$.

\begin{lemma}\cite[Lemma 4.2]{Daws} \label{banach dual = operator dual}
Let $X$ be a $p$-operator space, and let $\varphi\in X'$, the Banach dual of $X$. Then $\varphi$ is $p$-completely bounded as a map to $\mathbb{C}$. Moreover, $\|\varphi\|_{pcb}=\|\varphi\|$.
\end{lemma}

If $E=L_p(\mu)$ for some measure $\mu$ and $X\subseteq \mathcal{B}(E)=\mathcal{B}(L_p(\mu))$, then we say that $X$ is a $p$-\textit{operator space on} $L_p$ \textit{space}. These $p$-operator spaces are often easier to work with. For example, let $\kappa_X:X\to X''$ denote the canonical inclusion from a $p$-operator space $X$ into its second dual. Contrary to operator spaces, $\kappa_X$ is \textit{not} always $p$-completely isometric. Thanks to the following theorem by Daws, however, we can easily characterize those $p$-operator spaces with the property that the canonical inclusion is $p$-completely isometric.

\begin{prop}\cite[Proposition 4.4]{Daws} \label{Daws characterization}
Let $X$ be a $p$-operator space. Then $\kappa_X$ is a $p$-complete contraction. Moreover, $\kappa_X$ is a $p$-complete isometry if and only if $X\subseteq \mathcal{B}(L_p(\mu))$ $p$-completely isometrically for some measure $\mu$.
\end{prop}

\section{Non-existence of $p$-Arveson-Wittstock-Hahn-Banach Theorem}

Let $V\subseteq W$ be two operator spaces. Arveson-Wittstock-Hahn-Banach theorem asserts that every completely bounded map $\varphi:V\to \mathcal{B}(H)$ has a completely bounded extension $\tilde{\varphi}:W\to \mathcal{B}(H)$, where $H$ is a Hilbert space. For $p$-operator spaces, the following question naturally arises.

\begin{question}
Let $V\subseteq W$ be $p$-operator spaces and $E$ an $SQ_p$ space. Does every $p$-completely bounded map $\varphi:V\to \mathcal{B}(E)$ have a $p$-completely bounded extension $\tilde{\varphi}:W\to \mathcal{B}(E)$?
\end{question} 

To show that this question has a negative answer, let $p\neq 2$, and let $E$ and $\lpw$ such that $E$ is a Hilbert space embedding to $\lpw$. The existence of such $E$ and $\lpw$ is guaranteed by, for example, \cite[Proposition 8.7]{DefanFloret}. Let $J:E \hookrightarrow \lpw$ denote the isometric embedding, then we can view $E$ as a subspace of $\lpw$. 

\begin{lemma}
Let $J$ be as above. With $p$-operator space structure $E^c$ and $\lpw^c$, $J$ becomes a $p$-complete isometry.
\end{lemma}

\begin{proof}
From Example \ref{examplefirst}, we note that $M_n(E^c)\subseteq M_n(\bp(\mathbb{C}, E))=\bp(\ell_p^n, \ell_p^n(E))$. For $[\xi_{ij}]\in M_n(E^c)$, the norm is given by
$$\|[\xi_{ij}]\|^p=\sup\left\{\sum_{i=1}^n\left\|\sum_{j=1}^n\lambda_j\xi_{ij}\right\|_E^p: \lambda_j\in \mathbb{C},~\sum_{j=1}^n |\lambda_j|^p\leq 1\right\}.$$
Since $J$ is an isometry, 
$$\left\|J\left(\sum_{j=1}^n\lambda_j\xi_{ij}\right)\right\|_{\lpw}=\left\|\sum_{j=1}^n\lambda_j\xi_{ij}\right\|_E$$
and it follows that
\begin{eqnarray*}\|J_n([\xi_{ij}])\|^p&=&\sup\left\{\sum_{i=1}^n\left\|\sum_{j=1}^n\lambda_jJ(\xi_{ij})\right\|_{\lpw}^p: \lambda_j\in \mathbb{C},~\sum_{j=1}^n |\lambda_j|^p\leq 1\right\}\\
&=&\sup\left\{\sum_{i=1}^n\left\|J\left(\sum_{j=1}^n\lambda_j\xi_{ij}\right)\right\|_{\lpw}^p: \lambda_j\in \mathbb{C},~\sum_{j=1}^n |\lambda_j|^p\leq 1\right\} \\
&=& \|[\xi_{ij}]\|^p.
\end{eqnarray*}
\end{proof}

Let $\tilde{E}=\mathbb{C}\oplus_p E$. Let $\pi:\tilde{E}\to E$ denote the projection from $\tilde{E}$ onto $E$ and define $\varphi:E^c\to \mathcal{B}(\tilde{E})$ and $\psi:\mathcal{B}(\tilde{E})\to E^c$ by 
$$\varphi(\xi)=T_\xi,\quad T_\xi(\lambda\oplus_p e)=0\oplus_p \lambda \xi,\quad \lambda\in \mathbb{C},\quad e\in E$$
and
$$\psi(T)=\pi T(1\oplus_p 0), \quad T\in \mathcal{B}(\tilde{E})$$
(see the diagram below).
$$\xymatrix{ \lpw^c \ar@/^1pc/[rrd]^{\tilde{\varphi}}  & & \\ E^c \ar[u]^J \ar@<0.5ex>[rr]^\varphi & & \bp(\tilde{E})\ar@<0.5ex>[ll]^\psi}$$
It is then easy to check that $\varphi$ and $\psi$ are $p$-complete contractions with $\psi\circ\varphi=id_{E^c}$. Suppose that $\varphi:E^c\to \mathcal{B}(\tilde{E})$ extends to $\tilde{\varphi}:\lpw^c\to \mathcal{B}(\tilde{E})$. Define $P:\lpw^c\to E^c$ by $P=\psi\circ \tilde{\varphi}$, then it follows that $P$ is a $p$-completely contractive projection onto $E^c$, meaning that $E$ must be a $1$-complemented subspace of $\lpw$. This is, however, impossible, because it would imply that a Hilbert space $E$ is isometrically isomorphic to some $L_p$ space with $p\neq 2$.

\section{A predual of $\cbp(V,M_n)$}

In this section, we define a normed space structure on $\mathbb{M}_n(V)$ whose Banach dual is isometrically isomorphic to $\cbp(V,M_n)$.

%%%%%%%%%%%%%%%%%%%%%%%%%%%%%%%%%%

\begin{lemma} \label{p norm and p' norm comparison}
Let $1< p,p' < \infty$ with $1/p+1/p'=1$. Let
$\lambda=\{\lambda_j\}_{1\leq j\leq n}$ be a finite sequence in
$\mathbb{C}$. Then
$$\|\lambda\|_{\ell_p^n}\leq n^{|1/p-1/p'|}\cdot\|\lambda\|_{\ell_{p'}^n}.$$
\end{lemma}

\begin{proof}
There is nothing to prove if $p=p'=2$. If $p>p'$, then $\|\lambda\|_{\ell_p^n}\leq \|\lambda\|_{\ell_{p'}^n}
\leq n^{|1/p-1/p'|}\cdot\|\lambda\|_{\ell_{p'}^n}$ since
$n^{|1/p-1/p'|}\geq 1$. Finally, assume $1<p<p'$ and let
$q=\frac{p'}{p}>1$ and let $q'$ be the conjugate exponent to $q$. By
H\"{o}lder's inequality,
$$\|\lambda\|_{\ell_p^n}^p\leq \left(\sum_{j=1}^n |\lambda_j|^{pq}\right)^{1/q}\cdot n^{1/q'}
=\left(\sum_{j=1}^n |\lambda_j|^{p'}\right)^{p/p'}\cdot n^{1-p/p'}$$
and hence $\|\lambda\|_{\ell_p^n}\leq
n^{|1/p-1/p'|}\cdot\|\lambda\|_{\ell_{p'}^n}$.
\end{proof}

\begin{lemma} \label{finite version operator norm and matrix p norm comparison}
Let $\alpha=[\alpha_{ij}]\in \mathbb{M}_{n,r}$ and
$\beta=[\beta_{kl}]\in \mathbb{M}_{r,n}$. Let $1< p,p'< \infty$ with
$1/p+1/p'=1$. Then we have
$$\|\alpha\|_{\mathcal{B}(\ell_p^r,\ell_p^n)}\leq \|\alpha\|_{p'}\cdot n^{|1/p-1/p'|}\quad \text{and} \quad \|\beta\|_{\mathcal{B}(\ell_p^n,\ell_p^r)}\leq \|\beta\|_{p}\cdot n^{|1/p-1/p'|},$$ where
$$\|\alpha\|_{p'}=\left(\sum_{i=1}^n\sum_{j=1}^r |\alpha_{ij}|^{p'}\right)^{1/p'}\quad\text{and}\quad
\|\beta\|_{p}=\left(\sum_{k=1}^r\sum_{l=1}^n
|\beta_{kl}|^{p}\right)^{1/p}.$$
\end{lemma}

\begin{proof}
Suppose $\xi=\{\xi_j\}_{j=1}^r$ is a unit vector in $\ell_p^r$. For each $i$,
$1\leq i\leq n$, let $\eta_i=\left|\sum_{j=1}^r
\alpha_{ij}\xi_j\right|$, then by H\"{o}lder's inequality,
$\eta_i\leq \left(\sum_{j=1}^r|\alpha_{ij}|^{p'}\right)^{1/p'}$
and by Lemma \ref{p norm and p' norm comparison},
$$\left(\sum_{i=1}^n \eta_i^p\right)^{1/p}\leq n^{|1/p-1/p'|}\cdot \left(\sum_{i=1}^n \eta_i^{p'}\right)^{1/p'}\leq n^{|1/p-1/p'|}\cdot\|\alpha\|_{p'} $$
and hence we get $\|\alpha\|_{\mathcal{B}(\ell_p^r,\ell_p^n)}\leq
n^{|1/p-1/p'|}\cdot\|\alpha\|_{p'}$.
To prove the second inequality, let $\gamma=\beta^T\in \mathbb{M}_{n,r}$, the transpose of $\beta$. Then by the argument above we have
$$\|\gamma\|_{\mathcal{B}(\ell_{p'}^r, \ell_{p'}^n)}\leq \|\gamma\|_p\cdot n^{|1/p-1/p'|}.$$
Since $\|\gamma\|_{\mathcal{B}(\ell_{p'}^r, \ell_{p'}^n)}=\|\beta\|_{\mathcal{B}(\ell_p^n,\ell_p^r)}$ and
$\|\gamma\|_p=\|\beta\|_p$, we get the desired inequality.
\end{proof}

Let $V$ be a $p$-operator space. Fix $n\in \mathbb{N}$ and define $\|\cdot\|_{1,n}:\mathbb{M}_n(V)\to [0,\infty)$ by
\begin{equation} \label{1,n norm definition}\|v\|_{1,n}=\inf\{\|\alpha\|_{p'}\|w\|\|\beta\|_p:r\in \mathbb{N},\quad v=\alpha w \beta,\quad \alpha\in \mathbb{M}_{n,r},\quad \beta\in \mathbb{M}_{r,n},\quad w\in M_r(V)\},\end{equation}
where $\|\cdot\|_{p'}$ and $\|\cdot\|_p$ as in Lemma \ref{finite version operator norm and matrix p norm comparison}.

\begin{prop}
Suppose that $V$ is a $p$-operator space and $n\in \mathbb{N}$. Then $\|\cdot\|_{1,n}$ defines a norm on $\mathbb{M}_n(V)$.
\end{prop}

\begin{proof}
Suppose $v_1,v_2\in \mathbb{M}_n(V)$. Let $\epsilon>0$. For $i=1,2$, we can find $\alpha_i$, $\beta_i$, and $w_i$ such that $v_i=\alpha_i w_i \beta_i$ with $\|w_i\|\leq 1$ and
\begin{equation}\label{alpha beta in terms of v}\|\alpha_i\|_{p'}<\left(\|v_i\|_{1,n}+\epsilon\right)^{1/p'},\quad \|\beta_i\|_{p}<\left(\|v_i\|_{1,n}+\epsilon\right)^{1/p}.\end{equation}
Let
$$\alpha=[\alpha_1~ \alpha_2],\quad \beta=\left[\begin{array}{c} \beta_1 \\ \beta_2 \end{array}\right],\quad \text{and}\quad
w=\left[\begin{array}{cc} w_1 & \\ & w_2 \end{array}\right],$$
then $\|\alpha\|_{p'}^{p'}=\|\alpha_1\|_{p'}^{p'}+\|\alpha_2\|_{p'}^{p'}$, $\|\beta\|_{p}^{p}=\|\beta_1\|_{p}^{p}+\|\beta_2\|_{p}^{p}$, and $\|w\|\leq 1$. Since $v_1+v_2=\alpha w \beta$, it follows that
\begin{eqnarray*}
\|v_1+v_2\|_{1,n}&\leq& \|\alpha\|_{p'}\|\beta\|_{p}\\
(\text{Young's inequality})&\leq& \frac{\|\alpha\|_{p'}^{p'}}{p'}+\frac{\|\beta\|_{p}^{p}}{p}\\
&=&\frac{\|\alpha_1\|_{p'}^{p'}+\|\alpha_2\|_{p'}^{p'}}{p'}+\frac{\|\beta_1\|_{p}^{p}+\|\beta_2\|_{p}^{p}}{p}\\
(\text{by }(\ref{alpha beta in terms of v}))&<&\frac{\|v_1\|_{1,n}+\|v_2\|_{1,n}+2\epsilon}{p'}+\frac{\|v_1\|_{1,n}+\|v_2\|_{1,n}+2\epsilon}{p}\\
&=&\|v_1\|_{1,n}+\|v_2\|_{1,n}+2\epsilon.
\end{eqnarray*}
Since $\epsilon$ is arbitrary, we get $\|v_1+v_2\|_{1,n}\leq \|v_1\|_{1,n}+\|v_2\|_{1,n}$.\\
For any $c\in \mathbb{C}$, if $v=\alpha w \beta$, then we have $cv=\alpha (cw) \beta$ and hence $\|cv\|_{1,n}\leq \|\alpha\|_{p'}|c|\|w\|\|\beta\|_p$. Taking the infimum, we get
\begin{equation}\label{scalar part 1}
\|cv\|_{1,n}\leq |c|\|v\|_{1,n}.
\end{equation}
When $c\neq 0$, replacing $c$ by $1/c$ and $v$ by $cv$ in (\ref{scalar part 1}) gives
\begin{equation}\label{scalar part 2}
|c|\|v\|_{1,n}\leq \|cv\|_{1,n},
\end{equation}
so (\ref{scalar part 1}) together with (\ref{scalar part 2}) gives $\|cv\|_{1,n}=|c|\|v\|_{1,n}$, which is obviously true when $c=0$.\\
Finally, suppose $\|v\|_{1,n}=0$. To show that $v=0$, it suffices to show that
\begin{equation}\label{1,n is a norm}
\|v\|\leq n^{2|1/p-1/p'|}\cdot\|v\|_{1,n}.
\end{equation}
Indeed, if $v=\alpha w \beta$ with $\alpha \in \mathbb{M}_{n,r}$, $\beta\in \mathbb{M}_{r,n}$, and $w\in M_r(v)$, then
\begin{eqnarray*}
\|v\|&\leq&\|\alpha\|\|w\|\|\beta\|\\
(\text{by Lemma }\ref{finite version operator norm and matrix p norm comparison})&\leq&\|\alpha\|_{p'}\cdot n^{|1/p-1/p'|}\cdot\|w\|\cdot\|\beta\|_{p}\cdot n^{|1/p-1/p'|}\\
&=&n^{2|1/p-1/p'|}\cdot \|\alpha\|_{p'}\cdot \|w\| \cdot\|\beta\|_{p}.
\end{eqnarray*}
Taking the infimum, (\ref{1,n is a norm}) follows.
\end{proof}

For a $p$-operator space $V$, let $\mathcal{T}_n(V)$ denote the normed space $(\mathbb{M}_n(V), \|\cdot\|_{1,n})$. 

\begin{lemma} \label{TMdual}
For a $p$-operator space $V$, $\mathcal{T}_n(V)'=M_n(V')=\cbp(V,M_n)$ isometrically.
\end{lemma}

\begin{proof}
The second isometric isomorphism comes from the definition of the $p$-operator space structure on $V'$. We follow the idea as in \cite[\S 4.1]{EffrosRuan}. Let $f=[f_{ij}]\in M_n(V')=\cbp(V, M_n)$. Note that $$\|f\|=\sup\{\|\llangle f, \tilde{v} \rrangle \|: r\in \mathbb{N},~\tilde{v}=[\tilde{v}_{kl}]\in M_r(V),~\|\tilde{v}\|\leq 1 \}.$$
Let $D_{n\times r}^p$ denote the closed unit ball of $\ell_p^{n\times r}$, then
\begin{eqnarray*}
\|f\|&=&\sup\{\left|\left\langle \llangle f,\tilde{v}\rrangle \eta, \xi~\right\rangle\right|:r\in \mathbb{N},~\tilde{v}=[\tilde{v}_{kl}]\in M_r(V),~\|\tilde{v}\|\leq 1,~\eta\in D_{n\times r}^p,~\xi\in D_{n\times r}^{p'}\}\\
&=&\sup\left\{\left|\sum_{i,j,k,l} f_{ij}(\tilde{v}_{kl})\eta_{(j,l)}\xi_{(i,k)}\right|:r\in \mathbb{N},~\tilde{v}=[\tilde{v}_{kl}]\in M_r(V),~\|\tilde{v}\|\leq 1,~\eta\in D_{n\times r}^p,~\xi\in D_{n\times r}^{p'}\right\}\\
&=&\sup\left\{\left|\sum_{i,j=1}^n \left\langle f_{ij}, \sum_{k,l=1}^r \xi_{(i,k)}\tilde{v}_{kl}\eta_{(j,l)}\right\rangle\right|:r\in \mathbb{N},~\tilde{v}=[\tilde{v}_{kl}]\in M_r(V),~\|\tilde{v}\|\leq 1,~\eta\in D_{n\times r}^p,~\xi\in D_{n\times r}^{p'}\right\}.
\end{eqnarray*}
Note that $\sum_{k,l=1}^r \xi_{(i,k)}\tilde{v}_{kl}\eta_{(j,l)}$ is the $(i,j)$-entry of the matrix product $\alpha \tilde{v} \beta$, where
$$\alpha=\left[\begin{array}{ccc} \xi_{(1,1)} & \cdots & \xi_{(1,r)} \\ \vdots & \ddots & \vdots \\ \xi_{(n,1)} & \cdots & \xi_{(n,r)} \end{array}\right]\quad \text{and} \quad \beta=\left[\begin{array}{ccc} \eta_{(1,1)} & \cdots & \eta_{(n,1)} \\ \vdots & \ddots & \vdots \\ \beta_{(1,r)} & \cdots & \eta_{(n,r)} \end{array}\right],$$
so
\begin{eqnarray}\|f\|&=&\sup\left\{\left|\sum_{i,j=1}^n \left\langle f_{ij}, (\alpha\tilde{v}\beta)_{ij}\right\rangle\right|:\|\tilde{v}\|\leq 1,~\|\alpha\|_{p'}\leq 1,~\|\beta\|_p\leq 1\right\} \nonumber \\
&=&\sup\left\{\left|\left\langle f, v\right\rangle\right|:v=\alpha\tilde{v}\beta,~\|\tilde{v}\|\leq 1,~\|\alpha\|_{p'}\leq 1,~\|\beta\|_p\leq 1\right\} \nonumber \\
&=&\sup\left\{\left|\left\langle f, v\right\rangle\right|:\|v\|_{1,n}\leq 1\right\}. \label{f norm calculationi}
\end{eqnarray}
Define the scalar pairing $\Phi:M_n(V')\to \mathcal{T}_n(V)'$ by $f\mapsto \langle f,\cdot \rangle$, then from (\ref{f norm calculationi}) it follows that $\Phi$ is an isometric isomorphism.
\end{proof}

\begin{prop}\label{sufficient}
Let $V\subseteq W$ be $p$-operator spaces such that the inclusion $\mathcal{T}_n(V) \hookrightarrow \mathcal{T}_n(W)$ is isometric. Then every $p$-completely contractive map $\varphi:V\to \mathcal{B}(L_p(\Omega))$ has a completely contractive extension $\tilde{\varphi}:W\to \mathcal{B}(L_p(\Omega))$.
\end{prop}

\begin{proof}
Following \cite[Corollay 4.1.4, Theorem 4.1.5]{EffrosRuan}, it suffices to assume that $\mathcal{B}(L_p(\Omega))=\mathcal{B}(\ell_p^n)=M_n$. If the inclusion $i:\mathcal{T}_n(V) \hookrightarrow \mathcal{T}_n(W)$ is isometric, then by Lemma \ref{TMdual}, the adjoint $i':\cbp(W,M_n)\to \cbp(V,M_n)$, which is a restriction mapping, is an exact quotient mapping.
\end{proof}

%%%%%%%%%%%%%%%%%%%%%%%%%%%%%%%%%%%%%%%%%

\section{$\ell_p$-polar decomposition}

Let $V\subseteq W$ be $p$-operator spaces. By Proposition \ref{sufficient}, if the inclusion $\mathcal{T}_n(V) \hookrightarrow \mathcal{T}_n(W)$ is isometric, then every $p$-completely contractive map $\varphi:V\to \mathcal{B}(L_p(\Omega))$ has a completely contractive extension $\tilde{\varphi}:W\to \mathcal{B}(L_p(\Omega))$. In this section, we consider a condition on $W$ under which the inclusion $\mathcal{T}_n(V) \hookrightarrow \mathcal{T}_n(W)$ is isometric. Recall that the vector $p$-norm of $x=(x_1,\ldots,x_n)\in \mathbb{C}^n$ is defined by
$$\|x\|_p=\left(\sum_{i=1}^n |x_i|^p\right)^{1/p}.$$
If we identify $\mathbb{M}_{r,n}$ with $\mathcal{B}(\ell_2^n,\ell_2^r)$, the space of all bounded linear operators from $\ell_2^n$ to $\ell_2^r$, it is well known that every $\beta\in \mathbb{M}_{r,n}$ with $r\geq n$ has a \textit{polar decomposition}, that is, $\beta$ can be written as $\beta=\tau\beta_0$, where $\tau\in \mathbb{M}_{r,n}$ has orthonormal columns, that is, $\tau$ is an isometry, and $\beta_0\in \mathbb{M}_{n}$ is positive semidefinite \cite[\S 7.3]{MatrixAnalysis}. For $p\neq 2$ and $r\geq n$, regarding $\mathbb{M}_{r,n}$ as $\mathcal{B}(\ell_p^n,\ell_p^r)$, the space of all bounded linear operators from $\ell_p^n$ to $\ell_p^r$, we ask if there is an $\ell_p$-analogue of the polar decomposition. First of all, we need to define what we should mean by polar decomposition when $p\neq 2$, because, for example, if $T:\ell_p^n\to \ell_p^n$, then the adjoint $T'$ is from $\ell_{p'}^n$ to $\ell_{p'}^n$, where $1/p+1/p'=1$, and therefore $T'T$ is not defined, which in turn means we lose the concept of positive (semi)definiteness. We use the definition below as a natural $p$-analogue of the polar decomposition. 

\begin{defi} \label{defi of ppolar} Let $r\geq n$. We say that $\beta\in \mathbb{M}_{r,n}=\mathcal{B}(\ell_p^n,\ell_p^r)$ is \textit{$\ell_p$-polar decomposible} if there is an isometry $\tau\in \mathbb{M}_{r,n}$ and an operator $\beta_0 \in \mathbb{M}_{n}$ such that $\beta=\tau\beta_0$. In this case, we say that $\beta=\tau\beta_0$ is an \textit{$\ell_p$-polar decomposition} of $\beta$. The set of all full rank $\ell_p$-polar decomposible $r\times n$ matrices is denoted by $\mathbb{M}_{r,n}^{(p)}$.
\end{defi}

\begin{remark}\quad \label{ppolarremark}
\begin{enumerate}
\item If $r<n$, then there is no isometry in $\mathbb{M}_{r,n}=\mathcal{B}(\ell_p^n,\ell_p^r)$ and hence we only consider the case $r\geq n$ in Definition \ref{defi of ppolar}.
\item It is well known \cite[\S 0.4]{MatrixAnalysis} that $\operatorname{rank} AB\leq \min\{\operatorname{rank} A, \operatorname{rank} B\}$ whenever $AB$ is defined for matrices $A$ and $B$, so if $\beta=\tau\beta_0$ is an $\ell_p$-polar decomposition of a full rank $r\times n$ matrix $\beta$, then
$$n=\operatorname{rank} \beta \leq \min\{\operatorname{rank}\tau, \operatorname{rank} \beta_0\}\leq n$$
and it follows that $\operatorname{rank} \tau= \operatorname{rank} \beta_0 =n$. In particular, $\beta_0$ is nonsingular.
\item If $\beta=\tau\beta_0$ is an $\ell_p$-polar decomposition of $\beta$, then $\|\beta\|_p=\|\beta_0\|_p$, where $\|\cdot\|_p$ is as in Lemma \ref{finite version operator norm and matrix p norm comparison}.
\end{enumerate}
\end{remark}

To give a characterization of $\ell_p$-polar decomposible matrices, we begin with a characterization of isometries from $\ell_p^n$ to $\ell_p^r$.  Recall that for a vector $x=(x_1,\ldots, x_m)$, we define $\supp x$, the \textit{support} of $x$, by $\supp x=\{i:1\leq i\leq m,\quad x_i\neq 0\}$. 

\begin{lemma}\label{first char} Let $1< p <\infty$, $p\neq 2$, and $r\geq n$. Then $\tau:\ell_p^n \to \ell_p^r$ is an isometry if and only if the columns of $\tau$ have mutually disjoint supports with each vector $p$-norm equal to $1$.
\end{lemma}

\begin{proof}
Let $\tau_j=\left[\begin{array}{c} \tau_{1j} \\ \vdots \\ \tau_{rj} \end{array}\right]$ denote the $j^\text{th}$ column of an $r\times n$ matrix $\tau$. If $\tau_1,\ldots,\tau_n$ have mutually disjoint supports with each $p$-norm equal to $1$, then for any $x=(x_1,\ldots, x_n)\in \ell_p^n$, we get
\begin{eqnarray*}
\|\tau x\|_p^p &=& \sum_{i=1}^r \left|\sum_{j=1}^n\tau_{ij}x_j\right|^p = \sum_{k=1}^n \sum_{i\in \supp \tau_k} \left|\sum_{j=1}^n\tau_{ij}x_j\right|^p \\
&=& \sum_{k=1}^n \sum_{i\in \supp \tau_k} \left|\tau_{ik}x_k\right|^p = \sum_{k=1}^n \left|x_k\right|^p \sum_{i\in \supp \tau_k} |\tau_{ik}|^p \\
&=& \|x\|_p^p.
\end{eqnarray*}
Conversely, suppose $\tau:\ell_p^n \to \ell_p^r$ is an isometry. Since $\tau_j=\tau e_j$ for each $j$, where $e_j$ denotes the unit vector in $\ell_p^n$ whose only non-zero entry is $1$ at the $j^\text{th}$ place, it follows that $\tau_j$ is of norm $1$. To show that columns of $\tau$ have mutually disjoint supports, let $j\neq k$ and consider $e_j\pm e_k$ in $\ell_p^n$. Since $\|e_j\pm e_k\|_p=2^{1/p}$, we get
$\|\tau_j\pm \tau_k\|_p^p=2$ and the result follows from \cite[Lemma 15.7.23]{Royden}.
\end{proof}

\begin{remark} The result above remains true when $p=1$.
\end{remark}

Let $V$ be a $p$-operator space. For $v\in \mathbb{M}_n(V)$, we define
\begin{equation} \label{conditionP}\|v\|_{2,n}=\inf\{\|\alpha\|_{p'}\|w\|\|\beta\|_p:r\in \mathbb{N},\quad v=\alpha w \beta,\quad \alpha^T\in \mathbb{M}_{r,n}^{(p')},\quad \beta\in \mathbb{M}_{r,n}^{(p)},\quad w\in M_r(V)\},\end{equation}
where $\alpha^T$ denotes the transpose of $\alpha$ and 
$$\|\alpha\|_{p'}=\left(\sum_{i=1}^n\sum_{j=1}^r |\alpha_{ij}|^{p'}\right)^{1/p'}\quad\text{and}\quad
\|\beta\|_{p}=\left(\sum_{k=1}^r\sum_{l=1}^n
|\beta_{kl}|^{p}\right)^{1/p}.$$ 

\begin{prop}
Let $V\subseteq W$ be $p$-operator spaces. If $\|w\|_{2,n}=\|w\|_{1,n}$ for all $w\in \mathbb{M}_n(W)$, then the inclusion $\mathcal{T}_n(V) \hookrightarrow \mathcal{T}_n(W)$ is isometric.
\end{prop}

\begin{proof}
Let $v\in \mathbb{M}_n(V)$. It is clear that $\|v\|_{\mathcal{T}_n(W)}\leq \|v\|_{\mathcal{T}_n(V)}$. Suppose $\|v\|_{\mathcal{T}_n(W)}<1$, then by assumption, one can find $r\in \mathbb{N}$, $\alpha\in \mathbb{M}_{n,r}$, $\beta\in \mathbb{M}_{r,n}$, and $w\in M_r(W)$ such that $v=\alpha w \beta$, $\alpha^T\in \mathbb{M}_{r,n}^{(p')}$, $\beta\in \mathbb{M}_{r,n}^{(p)}$, $\|\alpha\|_{p'}<1$, $\|w\|<1$, and $\|\beta\|_p<1$. Let $\beta=\tau\beta_0$ (respectively, $\alpha^T=\sigma\alpha_0$) be $\ell_p$-(respectively, $\ell_p'$-) polar decomposition of $\beta$ (respectively, $\alpha^T$), and set $\tilde{w}=\sigma^T w \tau$, then $\|\tilde{w}\|_{M_n(W)}< 1$. Moreover, by Remark \ref{ppolarremark}, $\alpha_0$ and $\beta_0$ are invertible and hence $\tilde{w}=(\alpha_0^T)^{-1}v\beta_0^{-1}\in M_n(V)$, giving that $\|\tilde{w}\|_{M_n(V)}< 1$. Since $v=\alpha_0^T\tilde{w}\beta_0$, $\|\alpha_0^T\|_{p'}=\|\alpha\|_{p'}<1$, and $\|\beta_0\|_p=\|\beta\|_p<1$ by Remark \ref{ppolarremark}, it follows that $\|v\|_{\mathcal{T}_n(V)}<1$.
\end{proof}

For any $v\in \mathbb{M}_n(V)$, it is clear that $\|v\|_{1,n}\leq \|v\|_{2,n}$  At this moment of writing, we do not know of any nontrivial example of $p$-operator space $V$ such that $\|\cdot\|_{1,n}=\|\cdot\|_{2,n}$. It is not even clear whether $\|\cdot\|_{2,n}$ defines a norm on $\mathbb{M}_n(V)$ for some $p$-operator space $V$  (see Remark \ref{almostnorm}). However, thanks to Lemma \ref{first char}, we can give a characterization of $\ell_p$-polar decomposible matrices which may lead to finding a nontrivial example of $p$-operator spaces $V$ such that $\|v\|_{1,n}= \|v\|_{2,n}$ for all $v\in \mathbb{M}_n(V)$.

\begin{prop} \label{chardecom}
Let $1< p <\infty$, $p\neq 2$, and $r\geq n$. Then $\beta=\left[\begin{array}{c} \rule[1pt]{1em}{1pt}\quad u_1 \quad\rule[1pt]{1em}{1pt} \\ \vdots \\ \rule[1pt]{1em}{1pt}\quad u_r\quad\rule[1pt]{1em}{1pt}  \end{array}\right]\in \mathbb{M}_{r,n}=\mathcal{B}(\ell_p^n,\ell_p^r)$ is $\ell_p$-polar decomposible  if and only if there are $u_{j_1},u_{j_2},\ldots,u_{j_n}$, not necessarily distinct, such that each $u_i$ $(1\leq i\leq r)$ is a scalar multiple of $u_{j_k}$ for some $k$, $1\leq k\leq n$.
\end{prop}

\begin{proof}
Let $\beta=\left[\begin{array}{c} \rule[1pt]{1em}{1pt}\quad u_1 \quad\rule[1pt]{1em}{1pt} \\ \vdots \\ \rule[1pt]{1em}{1pt}\quad u_r \quad\rule[1pt]{1em}{1pt} \end{array}\right]\in \mathbb{M}_{r,n}=\mathcal{B}(\ell_p^n,\ell_p^r)$.  Suppose that there are $u_{j_1},u_{j_2},\ldots,u_{j_n}$ (not necessarily distinct) such that each $u_i$ $(1\leq i\leq r)$ is a scalar multiple of $u_{j_k}$ for some $k$, $1\leq k\leq n$. Rearranging rows of $\beta$ with an appropriate permutation if necessary, we may assume that $1= j_1<j_2<j_3<\cdots<j_n\leq r$ and that for $i$ with $j_k\leq i <j_{k+1}$, $u_i=c_{i} u_{j_k}$ for some scalar $c_{i}$. For each $k$, $1\leq k \leq n$, we define 
$\lambda_k=\left(\sum_{j_k\leq i < j_{k+1}}|c_{i}|^p\right)^{-p}.$ 
Note that $\lambda_k$ is well defined since $c_{j_k}=1$. Define $\tau \in \mathbb{M}_{r,n}$ and $\beta_0 \in \mathbb{M}_{n}$ by
$$\tau=\left[\begin{array}{ccccc} c_{1}\lambda_1 & 0 & 0 & \cdots & 0 \\ c_2\lambda_1 & 0 & 0 & \cdots & 0 \\
 \vdots&\vdots&\vdots&\ddots&\vdots \\
 c_{j_2-1}\lambda_1 & 0 & 0 & \cdots & 0 \\
 0 & c_{j_2}\lambda_2 & 0 & \cdots & 0 \\ 
  0 & c_{j_2+1}\lambda_2 & 0 & \cdots & 0 \\ 
   \vdots & \vdots & \vdots & \ddots & \vdots \\ 
   0 & c_{j_3-1}\lambda_2 & 0 & \cdots & 0 \\ 
      \vdots & \vdots & \vdots & \ddots & \vdots \\ 
   0 & 0 & 0 & \cdots & c_{j_n}\lambda_n \\ 
   0 & 0 & 0 & \cdots & c_{j_n+1}\lambda_n \\ 
     \vdots & \vdots & \vdots & \ddots & \vdots \\ 
      0 & 0 & 0 & \cdots & c_{r}\lambda_n 
\end{array}\right]\quad\text{and}\quad\beta_0=\left[\begin{array}{c} \rule[1pt]{1em}{1pt}\quad \frac{1}{\lambda_1}u_{j_1} \quad\rule[1pt]{1em}{1pt} \\  \rule[1pt]{1em}{1pt}\quad  \frac{1}{\lambda_2}u_{j_2} \quad\rule[1pt]{1em}{1pt} \\ \vdots \\ \rule[1pt]{1em}{1pt}\quad  \frac{1}{\lambda_n}u_{j_n} \quad\rule[1pt]{1em}{1pt} \end{array}\right],$$
then by Lemma \ref{first char}, it follows that $\beta=\tau\beta_0$ is an $\ell_p$-polar decomposition of $\beta$.\\
Conversely, assume that $\beta=\tau\beta_0$ is a $p$-polar decomposition of $\beta$. To exclude triviality, we may assume that $\beta$ contains no rows of only zeros.  Let $\tau_k$ denote the $k^\text{th}$ column of $\tau$. By Lemma \ref{first char}, $\supp \tau_k\neq \emptyset$ so we can pick $j_k\in \supp \tau_k$. Moreover, for each $i$, $1\leq i \leq r$, there is exactly one $k(i)$ such that $i\in \supp \tau_{k(i)}$ and it follows that $u_i$ is a constant multiple of $u_{j_{k(i)}}$. \end{proof}

\begin{remark} \label{almostnorm}
Let $v_1\in \mathbb{M}_n(V)$ and $v_2\in \mathbb{M}_m(V)$ for some $p$-operator space $V$, then one can easily show that $\|cv_1\|_{2,n}=|c|\|v_1\|_{2,n}$. Moreover, the decomposition $v_1=\alpha_1^T w_1 \beta_1$ and $v_2=\alpha_2^T w_2 \beta_2$ gives 
\begin{equation} \left[\begin{array}{cc} v_1 & \\ & v_2\end{array}\right]=\left[\begin{array}{cc} \alpha_1^T & \\ & \alpha_2^T\end{array}\right]\left[\begin{array}{cc}w_1 & \\ & w_2\end{array}\right]\left[\begin{array}{cc} \beta_1 & \\ & \beta_2\end{array}\right],\end{equation}
which, combined with Proposition \ref{chardecom}, shows that $\|v_1\oplus v_2\|_{2,n+m}\leq \|v_1\|_{2,n}+\|v_2\|_{2,m}$. 
\end{remark}

\vspace{1cm}
\bibliography{jungjinthesisbib}
\bibliographystyle{alpha}

\end{document}